\documentclass[letterpaper]{amsart}
\usepackage{amssymb}
\newtheorem{theorem}{Theorem}[section]
\newtheorem{lemma}[theorem]{Lemma}
\newtheorem{proposition}[theorem]{Proposition}
\newtheorem{corollary}[theorem]{Corollary}

\theoremstyle{definition}
\newtheorem{definition}[theorem]{Definition}

\theoremstyle{remark}
\newtheorem{remark}[theorem]{Remark}

\numberwithin{equation}{section}

\begin{document}

\title[Universal Sofic and Hyperlinear Groups]{Logic for Metric Structures and the Number\\ of Universal Sofic and Hyperlinear
Groups}


\author{Martino Lupini}
\address{Martino Lupini\\
Department of Mathematics and Statistics\\
N520 Ross, 4700 Keele Street\\
Toronto Ontario M3J 1P3, Canada, and Fields Institute for Research in
Mathematical Sciences\\
222 College Street\\
Toronto ON M5T 3J1, Canada.}
\email{mlupini@mathstat.yorku.ca}
\thanks{The research was supported by the York University Elia Scholars Program, the ESF Short Visit Grant No. 4154, the National University of Singapore and the John Templeton Foundation}


\date{}

\dedicatory{}



\begin{abstract}

Using the model theory of metric structures, we give an alternative proof of the following result by Thomas: If the Continuum Hypothesis fails, then there are $2^{2^{\aleph _{0}}}$ universal sofic groups up to isomorphism. This method is also applicable to universal hyperlinear groups, giving a positive answer to a question posed by Thomas. 


\end{abstract}

\maketitle

\section{Introduction}

Sofic and hyperlinear groups are two classes of discrete groups that have
received the attention of many mathematicians in different areas in the last
ten years. It is known that the class of sofic groups is contained in the
class of hyperlinear groups, but it is not known if this containment is
proper or whether the class of hyperlinear groups is equal to the class of
all discrete groups. For a complete presentation of this topic, the reader
is referred to \cite{Pestov1}. In \cite{Elek-Szabo1}, Elek and Szab\'{o}
proved that a countable group $\Gamma $ is sofic if and only if it can be
embedded in some (or, equivalently, every) ultraproduct of the symmetric
groups, regarded as a bi-invariant metric group with respect to the
normalized Hamming distance. (See \cite{BY-B-H-U} for the definition of
metric ultraproducts and an introduction to the logic for metric structures.) An
analogous characterization holds for hyperlinear groups, where the symmetric
groups are replaced with the finite rank unitary groups, endowed with the
normalized distance induced by the Hilbert-Schmidt norm. In view of this
characterization, metric ultraproducts of symmetric groups are said to be 
universal sofic groups, and metric ultraproducts of unitary groups are said to be 
universal hyperlinear groups. In \cite{Thomas}, Thomas proved
that if the Continuum Hypothesis fails, then there are $2^{\mathfrak{c}}$ many
metric ultraproducts of the symmetric groups up to (algebraic) isomorphism,
where $\mathfrak{c}$ denotes the cardinality of the continuum, and asked if
the same statement holds for ultraproducts of the unitary groups. In this
paper, we give a proof of Thomas' result, by means of the logic for metric
structures, which also applies in the case of the unitary groups. From this, I
deduce also the existence, under the failure of the Continuum Hypothesis, of 
$2^{\mathfrak{c}}$ many metric ultraproducts of the matrix algebras regarded
as ranked regular rings, up to algebraic isomorphism. This problem was
raised by G\'{a}bor Elek, in view of Proposition~8.3 in \cite{Farah-Shelah},
asserting the existence, under the failure of the Continuum hypothesis, of $%
2^{\mathfrak{c}}$ isomorphism classes of metric ultraproducts of the complex
matrix algebras, regarded as tracial von Neumann algebras.

Under the Continuum Hypothesis, the number of metric ultraproducts of
symmetric and unitary groups up to isomorphism is still unknown. In this
case, the statement that they are all isomorphic is equivalent to the
assertion that they are all elementarily equivalent as metric structures. At the end of this
article, we prove a partial result in this direction, showing that they have
the same $\Sigma _{2}$-theories as metric structures.

This article is organized as follows. In Section~\ref{Section: The order
property for symmetric and unitary groups}, we show that the sequences of
symmetric and unitary groups have the order property. This implies that,
under the failure of the Continuum Hypothesis, there are $2^{\mathfrak{c}}$
many ultraproducts of these sequences up to isometric isomorphism. In
Section~\ref{Section: Non-isomorphic universal sofic and hyperlinear groups}, we introduce some results and terminology from \cite{Farah-Shelah} in order
to deduce the existence, under the failure of the Continuum Hypothesis, of $%
2^{\mathfrak{c}}$ many ultraproducts up to algebraic isomorphism. In Section~\ref{Section: Ranked regular rings}, we infer from this result the existence,
under the failure of the Continuum Hypothesis, of $2^{\mathfrak{c}}$ many
ultraproducts of the sequence of matrix algebras up to algebraic
isomorphism. In Section~\ref{Section: The theories of universal sofic and
hyperlinear groups}, we prove that all universal sofic groups are
elementarily equivalent as metric structures with respect to $\Sigma _{2}$ formulas and that the
same holds for universal hyperlinear groups.

If $n\in \mathbb{N}$, the symmetric group acting on $\left\{ 1,\ldots
,n\right\} $ on the left is denoted by $S_{n}$ and its identity by $e_{n}$. The unitary
group of $n\times n$ matrices over $\mathbb{C}$ is denoted by $U_{n}$ and
its identity by $I_{n}$. The symmetric group $S_{n}$ is regarded as metric
group with respect to the bi-invariant metric defined by%
\begin{equation*}
d_{S_{n}}\left( \sigma ,\tau \right) =\frac{1}{n}\left\vert \left\{ i\in
\left\{ 1,\ldots ,n\right\} \left\vert \,\sigma \left( i\right) \neq \tau
\left( i\right) \right. \right\} \right\vert \text{,}
\end{equation*}%
called the normalized Hamming distance. The unitary group $U_{n}$ is endowed
with the metric%
\begin{equation*}
d_{U_{n}}\left( A,B\right) =\frac{\left\Vert A-B\right\Vert _{2}}{2\sqrt{n}}\text{,}
\end{equation*}%
where $\left\Vert \cdot \right\Vert _{2}$ denotes the Hilbert-Schmidt norm.
Usually the factor $\frac{1}{2}$ is omitted. It is introduced here only
because, in the logic for bounded metric structures, for convenience all the
metric spaces are supposed to have diameter at most $1$. By the universal sofic
and, respectively, hyperlinear groups, we will mean in the following the
metric ultraproducts of the sequences of the symmetric and, respectively,
unitary groups. Consider, for $n\in \mathbb{N}$, the injective homomorphism $%
\sigma \mapsto A_{\sigma }$ from $S_{n}$ to $U_{n}$ defined by%
\begin{equation}
A_{\sigma }\left( b_{i}\right) =b_{\sigma \left( i\right) }  \label{Eq: 1.1}
\end{equation}%
for $i\in \left\{ 1,2,\ldots ,n\right\} $, where $b_{1},\ldots ,b_{n}$ is
the canonical basis of $\mathbb{C}^{n}$, and observe that%
\begin{equation}
d_{U_{n}}\left( A_{\sigma },A_{\tau }\right) =\sqrt{\frac{d_{S_{n}}\left(
\sigma ,\tau \right) }{2}}  \label{Eq: 1.2}
\end{equation}%
(see \cite{Pestov1}).

In the rest of the paper, we will use the following notational conventions:
If $a,b$ are elements of a group $G$, then $\left[ a,b\right] $ denotes the
element $aba^{-1}b^{-1}$ of $G$. Uppercase calligraphic letters such as $\mathcal{U},\mathcal{V}$ are reserved for ultrafilters over 
$\mathbb{N}$. If $\left( M_{n}\right) _{n\in \mathbb{N}}$ is a sequence of
metric structures and $\mathcal{U}$ is an ultrafilter over $\mathbb{N}$, the
corresponding metric ultraproduct is denoted by $\prod_{n}^{\mathcal{U}%
}M_{n} $, while the ultrapower of a metric structure $M$ by $\mathcal{U}$ is
denoted by $M^{\mathcal{U}}$. We will denote by $\bar{x}$ and $\bar{y}$ $m$%
-tuples of variables $x_{1},\ldots ,x_{m}$ and $y_{1},\ldots ,y_{m}$. Every
metric will be denoted by $d$. The context will make clear which metric we are
referring to each time. The set of natural numbers $\mathbb{N}$ is supposed
not to contain $0$, and if $r$ is a real number, then $\left\lceil
r\right\rceil $ denotes the smallest integer greater than or equal to $r$. For convenience, we suppose $0$ to be a multiple of any natural number. Finally we will write, as usual, the acronym CH to stand for
\textquotedblleft Continuum Hypothesis".

\section{The order property for symmetric and unitary groups}

\label{Section: The order property for symmetric and unitary groups}

In \cite{Farah-Shelah}, Theorem~6.1, aiming to count the number of
ultrapowers of a $C^{\ast }$-algebra or of a von Neumann algebra, Farah and
Shelah isolate a condition ensuring that a sequence of metric structures has 
$2^{\mathfrak{c}}$ many ultraproducts up to isometric isomorphism, under the
failure of CH.

In this section, we will consider a particular case of \cite{Farah-Shelah},
Theorem~6.1, for bi-invariant metric groups, and we will infer from that the
following:

\begin{proposition}
\label{Proposition: non-isometrically isomorphic}If CH fails and $\left(
k_{n}\right) _{n\in \mathbb{N}}$ is a strictly increasing sequence of
natural numbers, then the sequences $\left( S_{k_{n}}\right) _{n\in \mathbb{N%
}}$ and $\left( U_{k_{n}}\right) _{n\in \mathbb{N}}$ have $2^{\mathfrak{c}}$
many ultraproducts up to isometric isomorphism.
\end{proposition}

In the following section, after introducing notation and definitions from 
\cite{Farah-Shelah}, we will refine this result, showing that in this case
there are in fact $2^{\mathfrak{c}}$ many ultraproducts up to algebraic
isomorphism, under the failure of CH.

The following proposition is a particular case of \cite{Farah-Shelah}, Theorem~6.1, obtained by considering the
language of bi-invariant metric groups and the formula 
\begin{equation*}
\varphi \left(x_{1},x_{2},y_{1},y_{2}\right) 
\end{equation*}
defined by
\begin{equation*}
d\left( \left[ x_{1},y_{2}\right] ,e\right) \text{.}
\end{equation*}

\begin{proposition}
\label{Proposition: Farah-Shelah proposition}Let $\left( k_{n}\right) _{n\in 
\mathbb{N}}$ be a strictly increasing sequence of natural numbers and $%
\left( G_{n}\right) _{n\in \mathbb{N}}$ be a sequence of bi-invariant metric
groups with uniformly bounded diameter. Suppose that, for some constant $%
\gamma >0$ and every $l\in \mathbb{N}$, for all but finitely many $n\in 
\mathbb{N}$, $G_{n}$ contains sequences $\left( g_{n,i}\right) _{i=1}^{l}$
and $\left( h_{n,i}\right) _{i=1}^{l}$ such that, for every $1\leq i<j\leq l$, $g_{n,i}$ and $h_{n,j}$ commute, while if $1\leq j\leq i\leq l$,%
\begin{equation*}
d\left( \left[ g_{n,i},h_{n,j}\right] ,e_{G_{n}}\right) \geq \gamma \text{.}
\end{equation*}%
Then, under the failure of CH, there are $2^{\mathfrak{c}}$ many metric ultraproducts of the sequence $\left(
G_{k_{n}}\right) _{n\in \mathbb{N}}$ up to isometric isomorphism.
\end{proposition}

Thus, in order to prove Proposition~\ref{Proposition: non-isometrically
isomorphic}, it is enough to show that the sequences of symmetric and
unitary groups satisfy the hypothesis of Proposition~\ref{Proposition:
Farah-Shelah proposition}.

\begin{lemma}
\label{Lemma: order property for symmetric groups 1}
For every $l\in \mathbb{N}$, there are sequences $\left( \sigma
_{l,i}\right) _{i=1}^{l},\left( \tau _{l,i}\right) _{i=1}^{l}$ of elements
of $S_{3^{l}}$ such that, if $1\leq i<j\leq l$, $\sigma _{l,i}$ and $\tau
_{l,j}$ commute, while if $1\leq j\leq i\leq l$, $\left[ \sigma _{l,i},\tau
_{l,j}\right] $ is the product of $3^{l-1}$ disjoint cycles of length $3$.
\end{lemma}

\begin{proof}
If $l\in \mathbb{N}$, consider the action of $\overset{l\text{ times}}{%
\overbrace{S_{3}\times \ldots \times S_{3}}}$ on $\left\{ 1,2,3\right\} ^{l}$
defined by%
\begin{equation*}
\left( \sigma _{1},\ldots ,\sigma _{l}\right) \left( i_{1},\ldots
,i_{l}\right) =\left( \sigma _{1}\left( i_{1}\right) ,\ldots ,\sigma
_{l}\left( i_{l}\right) \right) \text{,}
\end{equation*}%
where $\sigma _{1},\ldots ,\sigma _{l}\in S_{3}$ and $i_{1},\ldots ,i_{l}\in
\left\{ 1,2,3\right\} $. This action defines an isometric embedding%
\begin{equation*}
\left( \sigma _{1},\ldots ,\sigma _{l}\right) \rightarrow \alpha \left(
\sigma _{1},\ldots ,\sigma _{l}\right)
\end{equation*}%
of $\overset{l\text{ times}}{\overbrace{S_{3}\times \ldots \times S_{3}}}$
into the group of permutations of $\left\{ 1,2,3\right\} ^{l}$, which can be identified with $S_{3^{l}}$. Define also, for $i=1,\ldots ,l$,%
\begin{equation*}
\sigma _{l,i}=\alpha \left( \overset{i\text{ times}}{\overbrace{\left(
12\right) ,\ldots ,\left( 12\right) }},\overset{l-i\text{ times}}{\overbrace{%
e_{3},\ldots ,e_{3}}}\right)
\end{equation*}%
and%
\begin{equation*}
\tau _{l,i}=\alpha \left( \overset{i-1\text{ times}}{\overbrace{e_{3},\ldots
,e_{3}}},\left( 23\right) ,\overset{l-i\text{ times}}{\overbrace{%
e_{3},\ldots ,e_{3}}}\right) \text{.}
\end{equation*}%
Observe that, for $i<j$, 
\begin{equation*}
\left[ \sigma _{l,i},\tau _{l,j}\right] =e_{3^{l}}\text{,}
\end{equation*}%
while, for $i\geq j$,%
\begin{equation*}
\left[ \sigma _{l,i},\tau _{l,j}\right] =\alpha \left( \overset{j-1\text{
times}}{\overbrace{e_{3},\ldots ,e_{3}}},\left( 132\right) ,\overset{l-j%
\text{ times}}{\overbrace{e_{3},\ldots ,e_{3}}}\right) \text{.}\qedhere
\end{equation*}
\end{proof}

\begin{lemma}
\label{Lemma: order property symmetric groups}If $n,k,l\in \mathbb{N}$ and $%
r\in \mathbb{N}\cup \left\{ 0\right\} $ are such that $n=3^{l}k+r$ and $%
0\leq r<3^{l}$, then there are sequences $\left( \Sigma _{n,i}\right) _{i=1}^{l}$
and $\left( T_{n,i}\right) _{i=1}^{l}$ in $S_{n}$ such that, for $1\leq
i,j\leq l$, $\Sigma _{n,i}$ and $T_{n,i}$ commute if $i<j$, while $\left[
\Sigma _{n,i},T_{n,j}\right] $ is the product of $3^{l-1}k$ disjoint cycles
of length $3$ and, in particular,%
\begin{equation*}
d\left( \left[ \Sigma _{n,i},T_{n,j}\right] ,e\right) =\frac{3^{l}k}{3^{l}k+r%
}\geq \frac{k}{k+1}\geq \frac{1}{2}\text{,}
\end{equation*}
if $i\geq j$.
\end{lemma}

\begin{proof}
Consider the action of $S_{3^{l}}$ on $\left\{ 1,\ldots ,3^{l}\right\}
\times \left\{ 1,\ldots ,k\right\} $ defined by%
\begin{equation*}
\sigma \left( i,j\right) =\left( \sigma \left( i\right) ,j\right)
\end{equation*}%
for every $i\in \left\{ 1,\ldots ,3^{l}\right\} $ and $j\in \left\{ 1,\ldots
,k\right\} $. This defines an isometric embedding of $S_{3^{l}}$ into $%
S_{3^{l}k}$. Moreover, letting $S_{3^{l}k}$ act on the first $3^{l}k$
elements of $\left\{ 1,\ldots ,n\right\} $ defines an algebraic embedding of $%
S_{3^{l}k}$ into $S_{n}$. The composition $\Phi $ of these two embeddings is
an algebraic embedding of $S_{3^{l}}$ into $S_{n}$. For $1\leq i\leq l$, define%
\begin{equation*}
\Sigma _{n,i}=\Phi \left( \sigma _{l,i}\right)
\end{equation*}%
and%
\begin{equation*}
T_{n,i}=\Phi \left( \tau _{l,i}\right) \text{,}
\end{equation*}%
where $\sigma _{l,i}$ and $\tau _{l,i}$ are the elements of $S_{3^{l}}$ defined in Lemma~\ref{Lemma: order property for symmetric groups 1}. Then, if $1\leq i,j\leq l$,%
\begin{equation*}
\left[ \Sigma _{n,i},T_{n,j}\right] =\left[ \Phi \left( \sigma _{l,i}\right)
,\Phi \left( \tau _{l,j}\right) \right] =\Phi \left( \left[ \sigma
_{l,i},\tau _{l,j}\right] \right) .
\end{equation*}%
If $i<j$, $\left[ \sigma _{l,i},\tau _{l,j}\right] $ is the identity and,
hence, $\left[ \Sigma _{n,i},T_{n,j}\right] $ is the identity. If $i\geq j$,
then $\left[ \sigma _{l,i},\tau _{l,j}\right] $ is a product of $3^{l-1}$
disjoint $3$-cycles and, hence, $\left[ \Sigma _{n,i},T_{n,j}\right] $ is
the product of $3^{l-1}k$ disjoint $3$-cycles.
\end{proof}

\begin{lemma}
\label{Lemma: order property for unitary groups}If $n,k,l\in \mathbb{N}$
and $r\in \mathbb{N}\cup \left\{ 0\right\} $ are as in the statement of
Lemma~\ref{Lemma: order property symmetric groups}, then there are sequences 
$\left( b_{n,i}\right) _{i=1}^{l}$ and $\left( c_{n,i}\right) _{i=1}^{l}$ in 
$U_{n}$ such that $b_{n,i}$ and $c_{n,i}$ commute if $i<j$, while $d\left( %
\left[ b_{n,i},c_{n,j}\right] \right) \geq \frac{1}{2}$ if $i\geq j$.
\end{lemma}

\begin{proof}
Let $\sigma \mapsto A_{\sigma}$ be the embedding of $S_n$ into $U_n$ given by (\ref{Eq: 1.1}) and let
\begin{equation*}
b_{n,i}=A_{\Sigma _{n,i}}\quad \text{and}\quad c_{n,i}=A_{T_{n,i}}\text{,}
\end{equation*}%
where $\Sigma _{n,i},T_{n,i}\in S_n$ are the permutations given by Lemma~\ref{Lemma: order property symmetric groups}.
Observe that if $i<j$, then%
\begin{equation*}
\left[ \Sigma _{n,i},\Sigma _{n,j}\right] =e_{n}
\end{equation*}%
and hence
\begin{equation*}
\left[ A_{\Sigma _{n,i}},A_{T}{}_{n,j}\right] =A_{\left[ \Sigma
_{n,i},T_{n,j}\right] }=A_{e_{n}}=I_{n},
\end{equation*}%
while, if $i\geq j$,%
\begin{equation*}
d\left( \left[ \Sigma _{n,i},T_{n,i}\right] ,e_n\right) \geq \frac{1}{2}
\end{equation*}%
and, by (\ref{Eq: 1.2}),%
\begin{equation*}
d\left( \left[ A_{\Sigma _{n,i}},A_{T_{n,j}}\right] ,I_{n}\right) =\sqrt{%
\frac{d\left( \left[ \Sigma _{n,i},T_{n,j}\right] ,e_{n}\right) }{2}}\geq 
\frac{1}{2}.\qedhere
\end{equation*}
\end{proof}

Proposition~\ref{Proposition: non-isometrically isomorphic} is now an immediate consequence of Lemma~\ref{Lemma: order property symmetric groups} and Lemma~\ref{Lemma: order property for unitary groups}, together with Proposition~\ref{Proposition: Farah-Shelah proposition}.

\section{Non-isomorphic universal sofic and hyperlinear groups}

\label{Section: Non-isomorphic universal sofic and hyperlinear groups}In
this section, we will prove the following strengthening of Proposition~\ref{Proposition: non-isometrically
isomorphic}.

\begin{theorem}
\label{Theorem: main theorem}If CH fails and $\left( k_{n}\right) _{n\in 
\mathbb{N}}$ is an increasing sequence of natural numbers, then, up to
algebraic isomorphism, there are $2^{\mathfrak{c}}$ many ultraproducts of
both the sequence $\left( S_{k_{n}}\right) _{n\in \mathbb{N}}$ and the
sequence $\left( U_{k_{n}}\right) _{n\in \mathbb{N}}$.
\end{theorem}

This result has already been proved by Thomas in \cite{Thomas} for
permutation groups. \L ukasz Grabowski pointed out that the case of
unitary groups can be deduced from Proposition~8.3 in \cite{Farah-Shelah},
using the fact that non-isomorphic type II von Neumann algebras have non-isomorphic
unitary groups (\cite{Feldman}, Theorem~4) and that the unitary group of an
ultraproduct of finite von Neumann algebras is the ultraproduct of the unitary groups (\cite%
{Ge-Hadwin}, Proposition~2.1). In the following, we will give a direct proof
of this result by means of the logic for metric structures. This generalization
of the usual discrete logic is suitable to deal with structures endowed with a
metric. For an introduction to this topic, the reader is referred to \cite%
{BY-B-H-U}. Structures and formulas in the usual discrete logic can be
considered particular cases of metric structures and formulas, where the
distance is interpreted as the trivial discrete distance defined by $d\left(
x,y\right) =1$ iff $x\neq y$. Thus, definitions and theorems stated in the
setting of the logic for metric structures subsume the analogous definitions and
theorems for the usual discrete logic as a particular case. For the sake of
simplicity, all the languages are henceforth supposed without relation
symbols, apart from the metric. If $M$ is a structure in such a language $%
\mathcal{L}$, denote by $M_{alg}$ the $\mathcal{L}$-structure obtained from $M$ by
replacing the metric on $M$ by the trivial discrete metric.

We now have to introduce some notation and recall some results from \cite%
{Farah-Shelah}. If $\mathcal{L}$ is a language, $\psi \left( \bar{x},\bar{y}%
\right) $ is an $\mathcal{L}$-formula, $\varepsilon \geq 0$ and $M$ is an $%
\mathcal{L}$-structure, then the relation $\prec _{\psi ,\varepsilon }$ on $M^{k}$
is defined by%
\begin{equation*}
\bar{a}\prec _{\psi ,\varepsilon }\bar{b}\Leftrightarrow \left( \psi
^{M}\left( \bar{a},\bar{b}\right) \leq \varepsilon \wedge \psi ^{M}\left( 
\bar{b},\bar{a}\right) \geq 1-\varepsilon \right) \text{.}
\end{equation*}%
A chain in $M^{k}$ with respect to the relation $\prec _{\psi ,\varepsilon }$
will be called a \emph{$\left( \psi ,\varepsilon \right) $-chain} in $M$.
The relation $\prec _{\psi ,0}$ will be denoted by $\prec _{\psi }$ and a $%
\left( \psi ,0\right) $-chain will be called a $\psi $-chain. If $M$ is an $%
\mathcal{L}$-structure and $\varphi \left( \bar{x},\bar{y}\right) $ an $%
\mathcal{L}$-formula, a $\psi $-chain $\mathcal{C}$ is called \emph{$\left(
\aleph _{1},\psi \right) $-skeleton like} if, for every $\bar{a}\in M^{k}$
there exists a countable $\mathcal{C}_{\bar{a}}\subset \mathcal{C}$ such that,
for every $\bar{b},\bar{c}\in \mathcal{C}$ such that%
\begin{equation*}
\left\{ x\in \mathcal{C}_{\bar{a}}\left\vert \,\bar{b}\preceq _{\psi
}x\preceq _{\psi }\bar{c}\right. \right\} =\varnothing \text{,}
\end{equation*}%
one has 
\begin{equation*}
\psi ^{M}\left( \bar{a},\bar{b}\right) =\psi ^{M}\left( \bar{a},\bar{c}%
\right) \quad \text{and}\quad \psi ^{M}\left( \bar{b},\bar{a}\right) =\psi
^{M}\left( \bar{c},\bar{a}\right) \text{.}
\end{equation*}%
The notion of a $\psi $-chain and an $\left( \aleph _{1},\psi \right) $-skeleton
like $\psi $-chain in a discrete structure for a discrete formula $\psi $
are obtained from the previous ones, as a particular case.

The following statement is proved in \cite{Farah-Shelah} (Proposition~6.6).

\begin{lemma}
\label{Lemma: construction ultrafilters}If $\varphi \left( \bar{x},\bar{y}%
\right) $ is an $\mathcal{L}$-formula, $I$ is a linear order of cardinality $%
\mathfrak{c}$ and $\left( M_{n}\right) _{n\in \mathbb{N}}$ is a sequence of $%
\mathcal{L}$-structures such that, $\forall n\in \mathbb{N}$, $M_{n}$
contains a $\varphi $-chain of length $n$, then there is an ultrafilter $%
\mathcal{U}$ over $\mathbb{N}$ such that $\prod_{n}^{\mathcal{U}}M_{n}$
contains an $\left( \aleph _{1},\varphi \right) $-skeleton like $\varphi $-chain of order type $I$.
\end{lemma}

The same fact for discrete structures and formulas can be inferred from this lemma
as a particular case. The following definition, taken from \cite{Farah-Hart-Sherman1}, is of key importance for the proof of the main result.
\begin{definition}
\label{Definition: order property}
A sequence $\left( M_{n}\right) _{n\in \mathbb{N}}$ of $\mathcal{L}$-structures has the \emph{order property} witnessed by the $\mathcal{L}$-formula $\varphi \left( \bar{x},\bar{y}\right) $ if, for every $\varepsilon >0$ and $l\in \mathbb{N}$, $M_n$ contains a $\left( \varphi ,\varepsilon \right) $-chain of length $l$ for all but finitely many $n\in \mathbb{N}$.
\end{definition}

The assumption that a sequence $\left( M_{n}\right) _{n\in \mathbb{N}}$ of $\mathcal{L}$-structures has the order property is slightly weaker than the assumption on $\left( M_{n}\right) _{n\in \mathbb{N}}$ in Lemma~\ref{Lemma: construction ultrafilters}. Nonetheless, if a sequence $\left( M_{n}\right) _{n\in \mathbb{N}}$ has the order property, then the same conclusion as in 
Lemma~\ref{Lemma: construction ultrafilters} holds, namely, for every linear order $I$ of cardinality $\mathfrak{c}$, there is an ultrafilter $\mathcal{U}$ over $\mathbb{N}$ such that $\prod_{n}^{\mathcal{U}}M_{n}$ contains an $\left( \aleph _{1},\varphi \right) $-skeleton like $\varphi $-chain of order type $I$. This is easily seen via a suitable modification of the proof of Proposition~6.6 in \cite{Farah-Shelah}.

The connection between the number of non-isomorphic ultraproducts and the $%
\left( \aleph _{1},\varphi \right) $-skeleton like $\varphi $-chains is
given by the following lemma. It is proved in \cite{Farah-Shelah} (Proposition~3.14) in the setting of usual first order logic. As pointed out in Section~6.5 of the same paper, the proof can be easily adapted to the metric case.

\begin{lemma}
\label{Lemma: class of models}If CH fails, $\varphi \left( \bar{x},\bar{y}%
\right) $ is an $\mathcal{L}$-formula and $\mathcal{K}$ is a class of $%
\mathcal{L}$-structures such that, for every linear order $I$ of cardinality 
$\mathfrak{c}$, there is an element $M$ of $\mathcal{K}$ such that $M$
contains an $\left( \aleph _{1},\varphi \right) $-skeleton like $\varphi $-chain, then there are $2^{\mathfrak{c}}$ many pairwise non-isometrically
isomorphic $\mathcal{L}$-structures in $\mathcal{K}$.
\end{lemma}

The following lemma is useful when, as in our case, one is interested in
counting the number of metric ultraproducts up to algebraic isomorphism.

\begin{lemma}
\label{Lemma: metric and discrete skeleton}Suppose that $\varphi \left( \bar{%
x},\bar{y}\right) $ is an $\mathcal{L}$-formula and $\psi \left( \bar{x},%
\bar{y}\right) $ is a discrete $\mathcal{L}$-formula such that, for every $%
\mathcal{L}$-structure $M$, for every $\bar{a},\bar{b}\in M^{k}$, $\psi
\left( \bar{a},\bar{b}\right) $ holds in $M_{alg}$ if and only if $\varphi
^{M}\left( \bar{a},\bar{b}\right) =0$. If $M$ is an $\mathcal{L}$-structure
and $\mathcal{C}$ is an $\left( \aleph _{1},\varphi \right) $-skeleton like $%
\varphi $-chain in $M$, then $\mathcal{C}$ is an $\left( \aleph _{1},\psi
\right) $-skeleton like $\psi $-chain of the same order type in $M_{alg}$.
\end{lemma}

\begin{proof}
The hypothesis implies that $\prec _{\psi }$ in $M_{alg}^{k}$ refines $\prec
_{\varphi }$ in $M^{k}$. Thus, a $\varphi $-chain in $M$ is also a $\psi $-chain in $M_{alg}$ of the same order type. Moreover, suppose $\bar{a}\in
M^{k}$ and $\mathcal{C}_{\bar{a}}$ is as in the definition of $\left( \aleph
_{1},\varphi \right) $-skeleton like. If $\bar{b},\bar{c}\in \mathcal{C}$
are such that%
\begin{equation*}
\left\{ x\in \mathcal{C}_{\bar{a}}\left\vert \,\bar{b}\preceq _{\psi
}x\preceq _{\psi }\bar{c}\right. \right\} =\varnothing \text{,}
\end{equation*}%
then also%
\begin{equation*}
\left\{ x\in \mathcal{C}_{\bar{a}}\left\vert \,\bar{b}\preceq _{\varphi
}x\preceq _{\varphi }\bar{c}\right. \right\} =\varnothing \text{.}
\end{equation*}%
Hence,%
\begin{equation*}
\varphi ^{M}\left( \bar{a},\bar{b}\right) =\varphi ^{M}\left( \bar{a},\bar{c}%
\right) \quad \text{and}\quad \varphi ^{M}\left( \bar{b},\bar{a}\right)
=\varphi ^{M}\left( \bar{c},\bar{a}\right) \text{.}
\end{equation*}%
Since, by hypothesis, $M_{alg}\models \psi ^{M}\left( \bar{a},\bar{b}\right) 
$ is equivalent to $\varphi ^{M}\left( \bar{a},\bar{b}\right) =0$ and $%
M_{alg}\models \psi ^{M}\left( \bar{a},\bar{c}\right) $ is equivalent to $%
\varphi ^{M}\left( \bar{a},\bar{c}\right) =0$, one gets%
\begin{equation*}
M_{alg}\models \left( \psi ^{M}\left( \bar{a},\bar{b}\right) \leftrightarrow
\psi ^{M}\left( \bar{a},\bar{c}\right) \right) \text{.}
\end{equation*}%
In the same way,%
\begin{equation*}
M_{alg}\models \left( \psi ^{M}\left( \bar{b},\bar{a}\right) \leftrightarrow
\psi ^{M}\left( \bar{c},\bar{a}\right) \right)
\end{equation*}%
is deduced. Thus, $\mathcal{C}$ is an $\left( \aleph _{1},\psi \right) $-skeleton like $\psi $-chain in $M_{alg}$.
\end{proof}

\begin{remark}
\label{Remark: formulas}If $s,t$ are $\mathcal{L}$-terms and $q:\left[ 0,1%
\right] \rightarrow \left[ 0,1\right] $ is a continuous function such that $%
q\left( x\right) =0$ iff $x=0$, the formulas 
\begin{equation*}
\varphi \left( s,t\right) =\text{\textquotedblleft }q\left( d\left(
s,t\right) \right) \text{"}\quad \text{and}\quad \psi \left( s,t\right) =%
\text{\textquotedblleft }s=t\text{"}
\end{equation*}%
satisfy the hypothesis of the previous lemma.
\end{remark}

\begin{proposition}
\label{Proposition: algebraically nonisomorphic metric ultraproducts}Assume
that CH fails. If $\varphi \left( \bar{x},\bar{y}\right) $ and $\psi \left( 
\bar{x},\bar{y}\right) $ are as in Lem-\linebreak ma~\ref{Lemma: metric and discrete
skeleton}, $\left( M_{n}\right) _{n\in \mathbb{N}}$ is a sequence of $%
\mathcal{L}$-structures with the order property witnessed by $\varphi $, and $%
\left( k_{n}\right) _{n\in \mathbb{N}}$ is a strictly increasing sequence of
natural numbers, then the family of $\mathcal{L}$-structures

\begin{equation*}
\left\{ \left( \prod\nolimits_{n}^{\mathcal{U}}M_{k_{n}}\right) _{alg}\left|
\,\mathcal{U}\text{ is an ultrafilter over }\mathbb{N}\right. \right\} ,
\end{equation*}%
where $\prod\nolimits_{n}^{\mathcal{U}}M_{k_{n}}$ denotes the metric
ultraproduct, contains $2^{\mathfrak{c}}$ many pairwise non-isomorphic
elements. In other words, there are $2^{\mathfrak{c}}$ many metric
ultraproducts of the sequence $\left( M_{k_{n}}\right) _{n\in \mathbb{N}}$
up to algebraic isomorphism, i.e., up to bijections preserving all the
function symbols but not necessarily the distance.
\end{proposition}

\begin{proof}
Since every subsequence of $\left( M_{n}\right) _{n\in \mathbb{N}}$ has the
order property witnessed by $\varphi $, there is no loss of generality
in assuming $k_{n}=n$ for every $n\in \mathbb{N}$. By Lemma~\ref{Lemma:
construction ultrafilters}, for every linear order $I$ of cardinality $%
\mathfrak{c}$, there is an ultrafilter $\mathcal{U}$ such that $\prod_{n}^{%
\mathcal{U}}M_{n}$ has an $\left( \aleph _{1},\varphi \right) $-skeleton
like $\varphi $-chain $\mathcal{C}$ of order type $I$. By Lemma~\ref{Lemma:
metric and discrete skeleton}, $\mathcal{C}$ is also an $\left( \aleph
_{1},\psi \right) $-skeleton like $\psi $-chain of the same order type in $%
\left( \prod\nolimits_{n}^{\mathcal{U}}M_{n}\right) _{alg}$. Since this is
true for every linear order $I$ of cardinality $\mathfrak{c}$, the
conclusion follows from Lemma~\ref{Lemma: class of models}.
\end{proof}

Observe that in this result $2^{\mathfrak{c}}$ is the maximum number
possible, because it is the number of ultrafilters over $\mathbb{N}$. The
following result is an immediate consequence of Proposition~\ref{Proposition: algebraically
nonisomorphic metric ultraproducts} and Remark \ref{Remark: formulas}.

\begin{corollary}
\label{Corollary: algebraically non-isomorphic metric ultraproducts}If $%
\left( M_{n}\right) _{n\in \mathbb{N}}$ is a sequence of $\mathcal{L}$-structures with the order property witnessed by the $\mathcal{L}$-formula $%
q\left( d\left( s,t\right) \right) $, where $s$ and $t$ are terms and $q:%
\left[ 0,1\right] \rightarrow \left[ 0,1\right] $ is a continuous function
such that $q\left( x\right) =0$ iff $x=0$, then the conclusion of
Proposition~\ref{Proposition: algebraically nonisomorphic metric
ultraproducts} holds.
\end{corollary}

Now, in order to prove Theorem~\ref{Theorem: main theorem} it is enough to
show that the sequences $\left( S_{n}\right) _{n\in \mathbb{N}}$ and $\left(
U_{n}\right) _{n\in \mathbb{N}}$ have the order property, witnessed by a
formula $\varphi $ as in Corollary \ref{Corollary: algebraically
non-isomorphic metric ultraproducts}. Let $\eta \left(
x_{1},x_{2},y_{1},y_{2}\right) $ be the metric formula defined by%
\begin{equation*}
\min \left\{ 2d\left( \left[ x_{1},y_{2}\right] ,e\right) ,1\right\} \text{.}
\end{equation*}%
Fix $l\in \mathbb{N}\ $and $n\geq 3^{l}$, and consider the sequences $\left(
\Sigma _{n,i}\right) _{i=1}^{l}$ and $\left( T_{n,i}\right) _{i=1}^{l}$ in $%
S_{n}$ defined in Lemma~\ref{Lemma: order property symmetric groups} and the
sequences $\left( b_{n,i}\right) _{i=1}^{l}$ and $\left( c_{n,i}\right)
_{i=1}^{l}$ in $U_{n}$ defined in Lemma~\ref{Lemma: order property
for unitary groups}. It is not difficult to infer from Lemma~\ref{Lemma:
order property symmetric groups} and Lemma~\ref{Lemma: order
property for unitary groups} that
\begin{equation*}
\left( \left( \Sigma _{n,i},T_{n,i}\right) \right) _{i=1}^{l}\quad \text{and}%
\quad \left( \left( b_{n,i},c_{n,i}\right) \right) _{i=1}^{l}
\end{equation*}%
are $\eta $-chains of length $l$ in $S_{n}$ and $U_{n}$ respectively. An
application of Corollary \ref{Corollary: algebraically non-isomorphic metric
ultraproducts} concludes the proof of Theorem~\ref{Theorem: main theorem}.

\section{Ranked regular rings}

\label{Section: Ranked regular rings}If $n\in \mathbb{N}$, denote by $%
\mathbb{M}_{n}$ the algebra of $n\times n$ matrices over $\mathbb{C}$ and by 
$\mathrm{rk}$ the normalized rank on $\mathbb{M}_{n}$. Thus, if $A\in 
\mathbb{M}_{n}$, $\mathrm{rk}\left( A\right) $ is the rank of $A$ divided by 
$n$. In \cite{Elek}, Elek considered metric ultraproducts of the matrix
algebras over $\mathbb{C}$ with respect to the \emph{rank metric}
\begin{equation*}
d\left( a,b\right) =\mathrm{rk}\left( a-b\right)\text{.}
\end{equation*}%


If $\sigma \in S_{n}$, denote, as before, by $A_{\sigma }$ the permutation
matrix associated to $\sigma $, regarded as an element of $\mathbb{M}_{n}$. It is easily seen that
\begin{equation}
\mathrm{rk}\left( I_{n}-A_{\sigma }\right) =1-\frac{l}{n}\text{,}
\label{Eq: 4.1}
\end{equation}
where $l$ is the number of possibly trivial cycles of $\sigma $.

It can be deduced from Lemma~\ref{Lemma: order property symmetric groups} that, for any increasing sequence $\left( k_n\right) _{n\in \mathbb{N}}$ of natural numbers, the sequence $\left( \mathbb{M}_{k_{n}}\right) _{n\in \mathbb{N}}$ of matrix algebras endowed with the rank metric has the order property witnessed by the formula $\varphi \left(x_{1},x_{2},y_{1},y_{2}\right) $ defined by%
\begin{equation*}
\min \left\{ 3d\left( x_{1}y_{2},y_{2}x_{1}\right) ,1\right\}\text{.}
\end{equation*}%
The proof of this fact is left to the reader, being very similar to the proof of Lemma~\ref{Lemma: order property for unitary groups}, using (\ref{Eq: 4.1}) instead of (\ref{Eq: 1.1}). The following proposition, that answers a question of Elek, can now be obtained by direct application of Corollary \ref{Corollary: algebraically non-isomorphic metric ultraproducts}.

\begin{proposition}
\label{Proposition: non-isomorphic ranked ultraproducts}If CH fails, then
for every increasing sequence $\left( k_{n}\right) _{n\in \mathbb{N}}$ of
natural numbers there are $2^{\mathfrak{c}}$ many metric ultraproducts of
the sequence $\left( \mathbb{M}_{k_{n}}\right) _{n\in \mathbb{N}}$ with
respect to the rank metric whose multiplicative semigroups
are pairwise non-isomorphic.
\end{proposition}

Proposition~\ref{Proposition: non-isomorphic ranked ultraproducts} can in fact be generalized to direct sequences of finite sums of matrix algebras obtained from a Bratteli diagram and a harmonic function as in \cite{Elek}, Section~3. Moreover, the same holds without change and with the same proof if $\mathbb{C}$ is replaced by any other field or by any von Neumann regular ring $R$ endowed with a rank function $N$. The easy details are left to the interested reader. An exhaustive treatment of von Neumann
regular rings and ranked von Neumann regular rings can be found in \cite{Goodearl}.

\section{The $\Sigma _{2}$-theories of universal sofic and hyperlinear groups%
}

\label{Section: The theories of universal sofic and hyperlinear groups}

If CH holds, then every universal sofic group and every universal hyperlinear
group is saturated when regarded as a metric structure; hence two such
groups are isometrically isomorphic if and only if their metric theories coincide.
(Again assuming CH, it is currently not known whether there exist algebraically
non-isomorphic universal sofic groups or algebraically non-isomorphic universal hyperlinear
groups; in \cite{Thomas}, Thomas asked whether all the universal sofic groups
were elementarily equivalent when regarded as first order structures.) In this section,
we will prove the partial results that all universal sofic groups have the same
metric $\Sigma _2$-theories and that all universal hyperlinear groups have the same metric
$\Sigma _2$-theories. This is equivalent to the
statement that for any $\Sigma _{2}$ formula $\varphi $ in the language of
metric groups, the sequences of real numbers given by evaluation of $%
\varphi $ in the symmetric groups and, respectively, in the unitary groups
converge. Since a formula $\varphi $ is $\Pi _{2}$ iff $1-\varphi $ is $%
\Sigma _{2}$, this implies that universal sofic, and respectively
hyperlinear, groups have the same $\Pi _{2}$ theories as well.

The proof will make use of a general lemma, roughly asserting that if $\left( M_n\right) _{n\in \mathbb{N}}$ is a sequence of structures such that, given $m\in \mathbb{N}$, the elements of $M_n$, for $n$ large enough, can be arbitrarily well approximated by elements in the range of some approximate embedding of $M_m$ into $M_n$, then for every $\Sigma _2$ sentence $\varphi$ the sequence $\left( \varphi ^{M_n}\right) _{n\in \mathbb{N}}$ converges.
In order to precisely state and prove the lemma, we need to introduce the following terminology. A function $\iota :M\rightarrow N$ between structures in a metric language
is said to \emph{preserve all the function and relation symbols up to $\delta \geq
0$} if, for every $n$-ary function symbol $f$ and $a_{1},\ldots ,a_{n}\in M$,%
\begin{equation*}
d^{N}\left( f^{N}\left( \iota \left( a_{1}\right) ,\ldots ,\iota \left(
a_{n}\right) \right) ,\iota \left( f^{M}\left( a_{1},\ldots ,a_{n}\right)
\right) \right) \leq \delta
\end{equation*}
and for every $n$-ary relation symbol $R$ and $a_{1},\ldots ,a_{n}\in M$,
\begin{equation*}
\left\vert R^{N}\left( \iota \left( a_{1}\right) ,\ldots ,\iota \left(
a_{n}\right) \right) -R^{M}\left( a_{1},\ldots ,a_{n}\right) \right\vert
\leq \delta \text{.}
\end{equation*}
If $\varphi \left(x_1,\ldots ,x_n\right) $ is a formula, the function $\iota $ is said to preserve $\varphi $ up to $\delta $ if
\begin{equation*}
\left\vert \varphi ^{N}\left( \iota \left( a_{1}\right) ,\ldots ,\iota \left(
a_{n}\right) \right) -\varphi ^{M}\left( a_{1},\ldots ,a_{n}\right) \right\vert
\leq \delta \text{.}
\end{equation*}
Theorem~3.5 of \cite{BY-B-H-U} asserts that the interpretation of a formula $\varphi $ in any structure is uniformly continuous in every variable, with uniform continuity modulus independent from the other variables and from the structure. It follows that for every $\epsilon >0$ there exists $\delta >0$ such that every embedding $\iota $ of a structure $M$ into a structure $N$ preserving all the function and relation symbols up to $\delta $ also preserves $\varphi $ up to $\epsilon $.

Observe that, if $\psi \left(y,x\right) $ is a quantifier-free formula, then
\begin{equation*}
\inf_{x}\sup_{y}\psi \left( y,x\right)
\end{equation*}
is a $\Sigma _2$ sentence. Conversely, by \cite{BY-B-H-U}, Theorem~6.3, Proposition~6.6 and Proposition~6.9, the set of such sentences is dense in the set of all $\Sigma _2$ sentences. Thus, there is no loss of generality in considering only this type of $\Sigma _2$ sentence.
\begin{lemma}
\label{Lemma: convergence of sigma2}Assume that $\left( M_{n}\right) _{n\in 
\mathbb{N}}$ is a sequence of structures in a metric language $\mathcal{L}$.
Suppose that, $\forall \delta >0$, $\exists m_{0}\in \mathbb{N}$ such that, $%
\forall m\geq m_{0}$, $\exists k_{0}\in \mathbb{N}$ such that $\forall k\geq
k_{0}$, $\forall a\in M_k$ there exists an embedding $\iota _{m}^{k}$ (possibly depending on $a$) of $M_{m}$ into $M_{k}$
satisfying the following properties:
\begin{itemize}
\item $\iota _{m}^{k}$ preserves all the relation and function symbols up to 
$\delta $,

\item there is $\tilde{a}\in M_{m}$ such that $d\left( \iota
_{m}^{k}\left( \tilde{a}\right) ,a\right) <\delta $.
\end{itemize}
Then if $\psi $ is a quantifier-free $\mathcal{L}$-formula and $%
\varphi $ is the $\mathcal{L}$-formula%
\begin{equation*}
\inf_{x}\sup_{y_{1},\ldots ,,y_{n}}\psi \left( y_{1},\ldots ,y_{n},x\right) \text{%
,}
\end{equation*}%
then the sequence $\left( \varphi ^{M_{m}}\right) _{m\in \mathbb{N}}$
converges.
\end{lemma}

\begin{proof}
Fix $\varepsilon >0$ and define $\delta >0$ such that any embedding
preserving all the function and relation symbols up to $\delta $ preserves $%
\psi $ up to $\varepsilon $, and moreover $\delta <\omega \left( \varepsilon \right) $, where $\omega $ is a uniform continuity modulus for $\psi $ in the last variable. Consider $m_0,m\geq m_0,k_0$ and $k\geq k_0$ as in the statement. If $a\in M_k$, then there exists an embedding $i_{m}^{k}$ of $M_m$ into $M_k$ that preserves all the functions and relation symbols up to $\delta $ and such that $d\left( \iota
_{m}^{k}\left( \tilde{a}\right) ,a\right) <\delta $ for some $\tilde{a} \in M_m$. Hence,%
\begin{eqnarray*}
\sup_{y_{1},\ldots ,y_{n}\in M_{k}}\psi \left( y_{1},\ldots ,y_{n},a\right) &\geq
&\sup_{y_{1},\ldots ,y_{n}\in M_{k}}\psi \left( y_{1},\ldots ,y_{n},\iota
_{m}^{k}\left( \tilde{a}\right) \right) -\varepsilon \\
&\geq &\sup_{y_{1},\ldots ,y_{n}\in M_{m}}\psi \left( \iota _{m}^{k}\left(
y_{1}\right) ,\ldots ,\iota _{m}^{k}\left( y_{n}\right) ,\iota
_{m}^{k}\left( \tilde{a}\right) \right) -\varepsilon \\
&\geq &\sup_{y_{1},\ldots ,y_{n}\in M_{m}}\psi \left( y_{1},\ldots ,y_{n},%
\tilde{a}\right) -2\varepsilon \\
&\geq &\inf_{x\in M_{m}}\sup_{y_{1},\ldots ,y_{n}\in M_{m}}\psi \left(
y_{1},\ldots ,y_{n},x\right) -2\varepsilon \\
&=&\psi ^{M_{m}}-2\varepsilon \text{.}
\end{eqnarray*}%
Since this is true for every $a\in M_{k}$,%
\begin{equation*}
\varphi ^{M_{k}}=\inf_{x\in M_{k}}\sup_{y_{1},\ldots ,y_{n}\in M_{k}}\psi
\left( y_{1},\ldots ,y_{n},x\right) \geq \varphi ^{M_{m}}-2\varepsilon .
\end{equation*}%
Since this is true for every $k\geq k_{0}$,%
\begin{equation*}
\liminf_{k\rightarrow +\infty }\psi ^{M_{k}}\geq \psi ^{M_{m}}-2\varepsilon \text{,}
\end{equation*}%
and since this is true for every $m\geq m_{0}$,%
\begin{equation*}
\liminf_{k\rightarrow +\infty }\psi ^{M_{k}}\geq \limsup_{m\rightarrow
+\infty }\psi ^{M_{m}}-2\varepsilon .
\end{equation*}%
Finally, letting $\varepsilon $ go to $0$, one gets%
\begin{equation*}
\liminf_{k\rightarrow +\infty }\psi ^{M_{k}}\geq \limsup_{m\rightarrow
+\infty }\psi ^{M_{m}}\text{.}
\end{equation*}%
This concludes the proof.
\end{proof}

We will now prove that the sequence $\left( S_n\right) _{n\in \mathbb{N}}$ of symmetric groups satisfies the hypothesis of 
Lemma~\ref{Lemma: convergence of sigma2}. The key observation is that a permutation $\sigma \in S_{km}$ belongs to the image of some isometric embedding of $S_m$ into $S_{km}$ if, for every $l\geq 1$, the number $l$-cycles of $\sigma $ is a multiple of $k$ (here and in the following, we consider $0$ to be multiple of any natural number). Thus, it is enough to prove that one can ``chop up'' any permutation $\sigma \in S_{km}$, obtaining another permutation $\tilde{\sigma }$ close to $\sigma $ with the number of its $l$-cycles a multiple of $k$ for every $l\in \mathbb{N}$. Lemma~\ref{Lemma: chop up 1} and Lemma~\ref{Lemma: chop up 2} essentially show that any permutation is close to a permutation with only ``small'' cycles, while Lemma~\ref{Lemma: chop up 3} shows that a permutation with only small cycles is close to one with the number of its $l$-cycles a multiple of $k$.

In order to simplify the discussion, we will introduce the following notation: If $\sigma \in S_{n}$ and $l\geq 1$, define $C_{l}\left( \sigma \right) $ to be the set of cycles of $\sigma $ of length $l$ and $w\left( \sigma \right) $
the greatest $l$ such that $C_{l}\left( \sigma \right) $ is non-empty. In
particular, $C_{1}\left( \sigma \right) $ is the set of fixed points of $%
\sigma $.

\begin{lemma}
\label{Lemma: chop up 1}
For every $m,n\in \mathbb{N}$ such that $m|n$, if $\sigma \in S_{n}$, then
there exists $\tau \in S_{n}$ such that $w\left( \tau \right) \leq m$, $%
C_{1}\left( \tau \right) \supset C_{1}\left( \sigma \right) $ and $d\left(
\sigma ,\tau \right) \leq \frac{2}{m}$
\end{lemma}

\begin{proof}
Suppose $k\in \mathbb{N}$ is such that $n=km$. If $\sigma $ is a cycle of
length $n$, there is a product of $k$ cycles of length $m$ at distance $%
\frac{1}{m}$ from $\sigma $. Suppose $\sigma $ has no cycle of length $n$.
If $m<l<n$, then $l=\lambda m+\rho $ for some $0\leq \rho <m$ and $1\leq
\lambda <k$. Pick $c\in C_{l}\left( \sigma \right) $ and consider the
permutation $\sigma ^{\prime }$ obtained by $\sigma $ breaking up $c$ into $%
\lambda $ cycles of length $m$ and, if $\rho >0$, one cycle of length $\rho $%
. Thus, 
\begin{equation*}
km\cdot d\left( \sigma ,\sigma ^{\prime }\right) \leq \lambda +1.
\end{equation*}%
Consider the permutation $\tau $ obtained by $\sigma $ repeating this
process for any element of $\bigcup_{k>m}C_{k}\left( \sigma \right) $. Then, 
$w\left( \tau \right) \leq m$. Define, for $\lambda \in \left\{
1,2,\ldots ,k-1\right\} $ and $\rho \in \left\{ 0,1,\ldots ,m-1\right\} $, 
\begin{equation*}
n_{\lambda ,\rho }=\left\vert C_{\lambda m+\rho }\left( \sigma \right)
\right\vert .
\end{equation*}%
Observe that%
\begin{equation*}
\sum_{\lambda =1}^{k-1} \sum_{\rho =0}^{m-1}\lambda mn_{\lambda ,\rho }\leq
\sum_{\lambda =1}^{k-1}\sum_{\rho =0}^{m-1}\left( \lambda m+\rho \right)
n_{\lambda ,\rho }\leq n=km,
\end{equation*}%
and hence%
\begin{equation*}
\sum_{\lambda =1}^{k-1}\sum_{\rho =0}^{m-1}\lambda n_{\lambda ,\rho }\leq k.
\end{equation*}%
Now, we have%
\begin{equation*}
km\cdot d\left( \sigma ,\tau \right) \leq \sum_{\lambda =1}^{k-1}\sum_{\rho
=0}^{m-1}\left( \lambda +1\right) n_{\lambda ,\rho }\leq 2\sum_{\lambda
=1}^{k-1}\sum_{\rho =0}^{m-1}\lambda n_{\lambda ,\rho }\leq 2k \text{,}
\end{equation*}%
and hence $d\left( \sigma ,\tau \right) \leq \frac{2}{m}$.
\end{proof}

\begin{lemma}
\label{Lemma: chop up 2}
For every $m,n\in \mathbb{N}$ such that $m|n$ and $\beta \in \left(
0,1\right) $, if $\sigma \in S_{n}$, then there exists $\tau \in S_{n}$ such
that $C_{1}\left( \tau \right) \supset C_{1}\left( \sigma \right) $, $%
w\left( \tau \right) \leq \left\lceil m^{\beta }\right\rceil $ and $d\left(
\sigma ,\tau \right) \leq \frac{8}{m^{\beta }}$.
\end{lemma}

\begin{proof}
Define $k=\frac{n}{m}$ and $N=\left\lceil \frac{km}{m^{\beta }}\right\rceil \left\lceil m^{\beta
}\right\rceil $. Observe that%
\begin{equation*}
N\leq \left( \frac{km}{m^{\beta }}+1\right) \left( m^{\beta }+1\right) =km+%
\frac{km}{m^{\beta }}+m^{\beta }+1\leq 4km\text{.}
\end{equation*}%
If $\sigma \in S_{km}$, consider the element $\tilde{\sigma }$ of $S_{N}$
acting as $\sigma $ on $\left\{ 1,2,\ldots ,km\right\} $ and fixing $\left\{
km+1,\ldots ,N\right\} $ pointwise. By Lemma~\ref{Lemma: chop up 1}, there is $%
\tilde{\tau }\in S_{N}$ such that $C_{1}\left( \tilde{\tau }\right)
\supset C_{1}\left( \tilde{\sigma }\right) $, $w\left( \tilde{\tau }%
\right) \leq \left\lceil m^{\beta }\right\rceil $ and $d\left( \tilde{%
\sigma },\tilde{\tau }\right) \leq \frac{2}{\left\lceil m^{\beta
}\right\rceil }$. Now consider the element $\tau $ of $S_{km}$ obtained
restricting $\tilde{\tau }$ to $\left\{ 1,2,\ldots ,km\right\} $. Now, $%
w\left( \tau \right) \leq \left\lceil m^{\beta }\right\rceil $, $C_{1}\left(
\tau \right) \supset C_{1}\left( \sigma \right) $ and%
\begin{equation*}
d\left( \sigma ,\tau \right) =\frac{N}{km}d\left( \tilde{\sigma },%
\tilde{\tau }\right) \leq \frac{8}{m^{\beta }}.\qedhere
\end{equation*}
\end{proof}

\begin{lemma}
\label{Lemma: chop up 3}

For every $\beta >0$ there exists $m_{0}\in \mathbb{N}$ such that, if $m\geq
m_{0}$, $k\in \mathbb{N}$, $n=km$ and $\tau \in S_{n}$ is such that $w\left(
\tau \right) \leq \left\lceil m^{\beta }\right\rceil $, then there exists $\rho
\in S_{n}$ such that $w\left( \rho \right) \leq \left\lceil m^{\beta
}\right\rceil $, $C_{1}\left( \rho \right) \supset C_{1}\left( \sigma
\right) $, $d\left( \rho ,\tau \right) <m^{2\beta -1}$ and $C_{i}\left( \rho
\right) $ a multiple of $k$ for every $i\in \mathbb{N}$.
\end{lemma}

\begin{proof}

Pick $m_{0}\in \mathbb{N}$ such that, for $m\geq m_{0}$, 
\begin{equation*}
\sum_{i=1}^{\left\lceil m^{\beta }\right\rceil }i\leq \frac{m^{2\beta
}+3m^{\beta }+2}{2}\leq m^{2\beta }\text{.}
\end{equation*}
Suppose $m\geq m_0$. Define, for $i\in \left\{ 2,\ldots ,\left\lceil m^{\beta }\right\rceil
\right\} $, 
\begin{equation*}
\left\vert C_{i}\left( \tau \right) \right\vert =n_{i}=t_{i}k+r_{i},
\end{equation*}%
where $0\leq r_{i}<k$ and $t_i\geq 0$. Observe that%
\begin{equation*}
\sum_{i=2}^{\left\lceil m^{\beta }\right\rceil }ir_{i}\leq
k\sum_{i=2}^{\left\lceil m^{\beta }\right\rceil }i\leq km^{2\beta }.
\end{equation*}%
Consider the permutation $\rho $ obtained dropping $r_{i}$ $i$-cycles from $%
\sigma $ for every $i\in \left\{ 2,3,\ldots ,\left\lceil m^{\beta }\right\rceil
\right\} $. Observe that $\left\vert C_{i}\left( \rho \right) \right\vert
=kt_{i}$ for $i=2,3,\ldots ,\left\lceil m^{\beta
}\right\rceil $. Moreover,%
\begin{equation*}
\left\vert C_{1}\left( \rho \right) \right\vert =km-\sum_{i=2}^{\left\lceil
m^{\beta }\right\rceil }C_{i}\left( \rho \right) =\left(
m-\sum_{i=2}^{\left\lceil m^{\beta }\right\rceil }t_{i}\right) k
\end{equation*}%
and $\left\vert C_{i}\left( \rho \right) \right\vert =0$ for $i>\left\lceil
m^{\beta }\right\rceil $. Finally,%
\begin{equation*}
km\cdot d\left( \tau ,\rho \right) \leq \sum_{i=2}^{\left\lceil m^{\beta
}\right\rceil }ir_{i}\leq km^{2\beta }\text{,}
\end{equation*}%
and hence $d\left( \tau ,\rho \right) \leq m^{2\beta -1}$.
\end{proof}

\begin{proposition}
\label{Proposition: isometric embedding}
There exists $m_{0}\in \mathbb{N}$ such that, for every $m\geq m_{0}$ and $k\in 
\mathbb{N}$, if $\sigma \in S_{km}$, then there exists $\rho \in S_{km}$ such
that $C_{1}\left( \rho \right) \supset C_{1}\left( \sigma \right) $, $%
w\left( \rho \right) \leq \left\lceil \sqrt[3]{m}\right\rceil $, $d\left(
\rho ,\sigma \right) \leq \frac{9}{\sqrt[3]{m}}$ and $\rho =\Phi \left( 
\tilde{\rho }\right) $ for some $\tilde{\rho }\in S_{m}$ and
isometric embedding $\Phi :S_{m}\rightarrow S_{km}$.
\end{proposition}

\begin{proof}
By Lemma~\ref{Lemma: chop up 2}, there is $\tau \in S_{km}$ such that $d\left( \tau
,\sigma \right) <\frac{8}{\sqrt[3]{m}}$, $w\left( \tau \right) \leq
\left\lceil \sqrt[3]{m}\right\rceil $ and $C_{1}\left( \tau \right) \supset
C_{1}\left( \sigma \right) $. By Lemma~\ref{Lemma: chop up 3}, there is $\rho \in
S_{km}$ such that $C_{1}\left( \rho \right) \supset C_{1}\left( \tau \right) 
$, $d\left( \rho ,\sigma \right) \leq \frac{1}{\sqrt[3]{m}}$, $w\left( \rho
\right) \leq \left\lceil \sqrt[3]{m}\right\rceil $ and $\left\vert C_{i}\left(
\sigma \right) \right\vert $ is a multiple of $k$ for every $i\in \mathbb{N}$. Thus, $\rho =\Phi \left( 
\tilde{\rho }\right) $ for some $\tilde{\rho }\in S_{m}$ and
isometric embedding $\Phi :S_{m}\rightarrow S_{km}$. Finally,%
\begin{equation*}
d\left( \sigma ,\rho \right) \leq d\left( \tau ,\rho \right) +d\left( \tau
,\sigma \right) \leq \frac{1}{\sqrt[3]{m}}+\frac{8}{\sqrt[3]{m}}\leq \frac{9%
}{\sqrt[3]{m}}.\qedhere
\end{equation*}
\end{proof}

If $k,m\in \mathbb{N}$ and $0\leq r<m$, define $\iota $ to be the injective group homomorphism of $S_{km}$ into $S_{km+r}$ obtained by sending a permutation $\sigma $ to the permutation that acts as $\sigma $ on $\left\{ 1,2,\ldots ,km\right\} $ and fixes pointwise $\left\{ km+1,
\ldots ,km+r\right\} $. Since $\iota $ preserves distances up to $1-\frac{1}{m}$, and any element of $S_{km+r}$ is at distance at most $1-\frac{1}{m}$ from some element in the range of $\iota $, it follows from Proposition~\ref{Proposition: isometric embedding} that the sequence $\left( S_{n}\right)
_{n\in \mathbb{N}}$ satisfies the hypothesis of Lemma~\ref{Lemma:
convergence of sigma2}. This concludes the proof of:

\begin{theorem}
If $\varphi $ is a formula of the form%
\begin{equation*}
\inf_{x}\sup_{y_{1},\ldots ,y_{n}}\varphi \left( y_{1},\ldots ,y_{n},x\right)
\end{equation*}%
in the language of bi-invariant metric groups, then the sequence $\left(
\varphi ^{S_{n}}\right) _{n\in \mathbb{N}}$ converges.
\end{theorem}

The analogue of Proposition~\ref{Proposition: isometric embedding} in the case of unitary groups has been proved by von Neumann in \cite{von_Neumann}, using the spectral theorem for normal matrices and an averaging argument on the eigenvalues. The precise statement is reported here for convenience of the reader. Observe that if $W\in U_{km}$, then the function from $U_{m}$ to $U_{km}$ sending $B$ to
\begin{equation*}
W\left( I_{k}\otimes B\right) W^{\ast }
\end{equation*}%
is an isometric embedding. Here, $\otimes $ denotes the usual tensor product of
matrices.

\begin{proposition}
\label{Proposition: von Neumann}If $\varepsilon >0$, there exists $m_{0}\in 
\mathbb{N}$ such that, for every $k\in \mathbb{N}$ and $m\geq n_{0}$, if $%
A\in \mathbb{M}_{km}$ is a normal matrix with operator norm at most $1$,
there exists $B\in M_{m}$ of operator norm at most $1$ and $W\in U_{km}$ such
that $\left\Vert A-W\left( I_{k}\otimes B\right) W^{\ast }\right\Vert
_{2}<\varepsilon $, where $\left\Vert \cdot \right\Vert _{2}$ is the normalized
Hilbert-Schmidt norm. Moreover, if $A$ is Hermitian (resp. unitary), $B$ can
be chosen Hermitian (resp. unitary).
\end{proposition}

In order to show that the sequence of unitary groups satisfies the hypothesis of Lemma~\ref{Lemma:
convergence of sigma2} it remains only to show that if $k,m\in \mathbb{N}$ and $0\leq r<k$, then there is an injective group homomorphism $\iota $ of $U_{km}$ into $U_{km+r}$ that almost preserves the metric and such that any element of $U_{km+r}$ is close to some element in the range of $\iota $. This is done in the following lemma.

\begin{lemma}
\label{Lemma: embedding}If $k,m\in \mathbb{N}$ and $0\leq r<m $, then the function $\iota $ from $U_{km}$ to $U_{km+r}$
sending $A$ to%
\begin{equation*}
\begin{pmatrix}
A & 0 \\ 
0 & I_{r}%
\end{pmatrix}%
\end{equation*}%
(where $I_{r}$ is the $r\times r$ identity matrix) is an injective group homomorphism
that preserves the metric up to $\frac{1}{k}$. Moreover, any element of $U_{km+r}$ is at distance at most $\frac{4}{\sqrt[4]{k}}$ from some element in the range of $\iota $.
\end{lemma}

\begin{proof}
Denote $km+r$ by $n$. Direct calculation shows that if $A,B\in U_{km}$,%
\begin{equation*}
0\leq d\left( A,B\right) -d\left( \iota \left( A\right) ,\iota \left(
B\right) \right) =1-\sqrt{\frac{km}{n}}\leq \frac{1}{k}.
\end{equation*}%
Suppose now that $C\in U_{n}$ and define $A$ to be the element of $\mathbb{M}%
_{km}$ such that $A_{i,j}=C_{i,j}$ for $1\leq i,j\leq km$. It is easy to see that, since $U$ is
unitary,%
\begin{equation*}
\left\Vert A^{\ast }A-I_{km}\right\Vert _{2}^{2}\leq \frac{1}{k}\text{.}
\end{equation*}%
By \cite{Glebsky}, Corollary 1, there exists $B\in U_{km}$ such that%
\begin{equation*}
\left\Vert A-B\right\Vert _{2}^{2}\leq 36\left\Vert A^{\ast
}A-I_{km}\right\Vert _2\leq \frac{36}{\sqrt{k}}\text{.}
\end{equation*}%
Thus,%
\begin{equation*}
d\left( \iota \left( B\right) ,C\right) \leq \frac{1}{2}\sqrt{\frac{%
km\left\Vert A-B\right\Vert _{2}^{2}+6r}{n}}\leq \frac{4}{\sqrt[4]{k}}\text{.%
}\qedhere
\end{equation*}
\end{proof}

This concludes the proof of:
\begin{theorem}
If $\varphi $ is a formula of the form%
\begin{equation*}
\inf_{x}\sup_{y_{1},\ldots ,y_{n}}\varphi \left( y_{1},\ldots ,y_{n},x\right)
\end{equation*}%
in the language of bi-invariant metric groups, then the sequence $\left(
\varphi ^{U_{n}}\right) _{n\in \mathbb{N}}$ converges.
\end{theorem}

\section*{Acknowledgements}
The author would like to thank his supervisor, Ilijas Farah, for help and support,
Samuel Coskey and Nicola Watson for their comments and suggestions, and \L ukasz Grabowski, Bradd
Hart, Ita\"{\i} Ben Yacoov, Ferenc Bencs, Louis-Philippe Thibault and Nigel
Sequeira for many useful conversations.


\end{document}